\documentclass[12pt]{amsart}
\usepackage{verbatim}
\usepackage{amssymb, amsmath}
\usepackage{graphicx}
\usepackage{color}
\usepackage{url}

\usepackage[shortlabels]{enumitem}

\begin{document}
\newtheorem{thm}{Theorem}[section]
\newtheorem*{thm*}{Theorem}
\newtheorem{lem}[thm]{Lemma}
\newtheorem{prop}[thm]{Proposition}
\newtheorem{cor}[thm]{Corollary}
\newtheorem*{conj}{Conjecture}
\newtheorem{proj}[thm]{Project}

\theoremstyle{definition}
\newtheorem*{defn}{Definition}
\newtheorem*{remark}{Remark}
\newtheorem{exercise}{Exercise}
\newtheorem*{exercise*}{Exercise}

\numberwithin{equation}{section}

\newcommand{\rad}{\operatorname{rad}}

\newcommand{\Z}{{\mathbb Z}} 
\newcommand{\Q}{{\mathbb Q}}
\newcommand{\R}{{\mathbb R}}
\newcommand{\C}{{\mathbb C}}
\newcommand{\Cc}{{\mathcal C}}
\newcommand{\N}{{\mathbb N}}
\newcommand{\FF}{{\mathbb F}}
\newcommand{\fq}{\mathbb{F}_q}
\newcommand{\rmk}[1]{\footnote{{\bf Comment:} #1}}

\renewcommand{\mod}{\;\operatorname{mod}}
\newcommand{\ord}{\operatorname{ord}}
\newcommand{\TT}{\mathbb{T}}
\renewcommand{\i}{{\mathrm{i}}}
\renewcommand{\d}{{\mathrm{d}}}
\renewcommand{\^}{\widehat}
\newcommand{\HH}{\mathbb H}
\newcommand{\Vol}{\operatorname{vol}}
\newcommand{\area}{\operatorname{area}}
\newcommand{\tr}{\operatorname{tr}}
\newcommand{\norm}{\mathcal N} 
\newcommand{\intinf}{\int_{-\infty}^\infty}
\newcommand{\ave}[1]{\left\langle#1\right\rangle} 
\newcommand{\Var}{\operatorname{Var}}
\newcommand{\Prob}{\operatorname{Prob}}
\newcommand{\sym}{\operatorname{Sym}}
\newcommand{\disc}{\operatorname{disc}}
\newcommand{\CA}{{\mathcal C}_A}
\newcommand{\cond}{\operatorname{cond}} 
\newcommand{\lcm}{\operatorname{lcm}}
\newcommand{\Kl}{\operatorname{Kl}} 
\newcommand{\leg}[2]{\left( \frac{#1}{#2} \right)}  
\newcommand{\Li}{\operatorname{Li}}

\newcommand{\sumstar}{\sideset \and^{*} \to \sum}

\newcommand{\LL}{\mathcal L} 
\newcommand{\sumf}{\sum^\flat}
\newcommand{\Hgev}{\mathcal H_{2g+2,q}}
\newcommand{\USp}{\operatorname{USp}}
\newcommand{\conv}{*}
\newcommand{\dist} {\operatorname{dist}}
\newcommand{\CF}{c_0} 
\newcommand{\kerp}{\mathcal K}

\newcommand{\Cov}{\operatorname{cov}}
\newcommand{\Sym}{\operatorname{Sym}}

\newcommand{\Ht}{\operatorname{Ht}}

\newcommand{\E}{\operatorname{\mathbb E}} 
\newcommand{\sign}{\operatorname{sign}} 
\newcommand{\meas}{\operatorname{meas}} 
\newcommand{\diam}{\operatorname{diam}} 
\newcommand{\length}{\operatorname{length}} 

\newcommand{\divid}{d} 

\newcommand{\GL}{\operatorname{GL}}
\newcommand{\SL}{\operatorname{SL}}
\newcommand{\re}{\operatorname{Re}}
\newcommand{\im}{\operatorname{Im}}
\newcommand{\res}{\operatorname{Res}}
 \newcommand{\eigen}{\Lambda} 

\newcommand{\legP}{ {\mathrm P}} 
\newcommand{\legQ}{ {\mathrm Q}}
\newcommand{\Ohyp}{ {\mathrm F}} 
\newcommand{\dsE}{\mathcal E} 

\title[Robin spectrum for the hemisphere]{On the Robin spectrum for the hemisphere}
\author{Ze\'ev Rudnick and Igor Wigman}
\address{School of Mathematical Sciences, Tel Aviv University, Tel Aviv 69978, Israel} \email{rudnick@tauex.tau.ac.il}
\address{Department of Mathematics, King's College London, UK}\email{igor.wigman@kcl.ac.uk}

\thanks{This research was supported by the European Research Council (ERC) under the European Union's  Horizon 2020 research and innovation programme  (Grant agreement No. 786758) and by the Israel Science Foundation (grant No. 1881/20). We are grateful to Nadav Yesha for discussions on various aspects of this project.}

\begin{abstract}

We study the spectrum of the Laplacian on the hemisphere with Robin boundary conditions. It is found that the eigenvalues fall into small clusters around the Neumann spectrum, and satisfy a Szeg\H{o} type limit theorem. 
Sharp upper and lower bounds for the gaps between the Robin and Neumann eigenvalues are derived, showing in particular that these are unbounded. 
Further, it is shown that except for a systematic double multiplicity, there are no multiplicities in the spectrum as soon as the Robin parameter is positive, unlike the Neumann case which is highly degenerate.
Finally, the limiting spacing distribution of the desymmetrized spectrum 
is proved to be the delta function at the origin.
 
\end{abstract}
\date{\today}
\maketitle


 \section{Introduction}

\subsection{The Robin problem}
Let $\Omega $ be the upper unit hemisphere (Figure~\ref{fig:hemisphere}), with its boundary $\partial \Omega$ the equator.
 Our goal is to study the Robin boundary problem on the hemisphere $\Omega$:
$$
\Delta F+\lambda F=0, \quad  \frac{\partial F}{\partial n} +\sigma F=0
$$
 where $\partial/\partial n$ is the derivative in the direction of the outward pointing normal to the equator, and $\sigma\geq 0$ is a constant. 
     \begin{figure}[ht]
\begin{center}
\includegraphics[height=50mm]{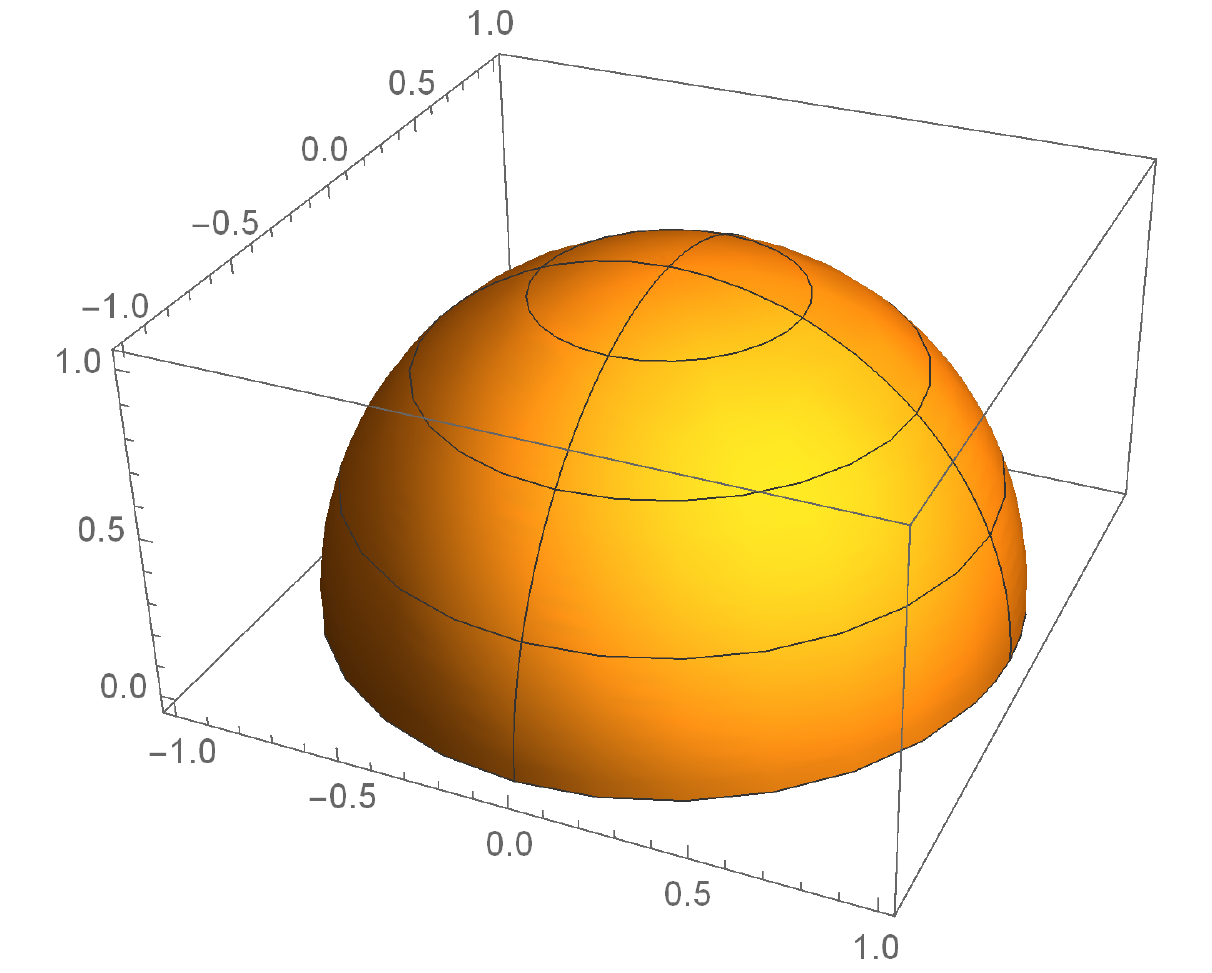}
\caption{ The hemisphere. }
\label{fig:hemisphere}
\end{center}
\end{figure}

The cases of Neumann and Dirichlet boundary conditions ($\sigma=0$ or $\sigma=\infty$) are classical \cite[p. 243-244]{BB}:
The eigenfunctions are restrictions to $\Omega$ of the eigenfunctions on the sphere (spherical harmonics), determined by the parity under reflection in the equator: The odd spherical harmonics give the Dirichlet eigenfunctions, the even ones give the Neumann eigenfunctions.
The eigenvalues are thus of the form $\ell(\ell+1)$, where $\ell\geq 0$ is an integer, repeated with multiplicity  $\ell+1$  for the Neumann case, and $\ell$ for the Dirichlet case.

The Robin spectrum is significantly less understood, and it is the main object of our interest.
The problem admits separation of variables, and there is a basis of eigenfunctions in the form
$f_{\nu,m}\linebreak[4] =e^{im\phi}\legP_\nu^m(\cos \theta)$, $m\in \Z$, where $\legP_\nu^m(x)$ is an associated Legendre function. For each $m$, the admissible $\nu$'s are determined by the boundary condition. 

Both $f_{\nu,m}$ and $f_{\nu,-m}$ share the same   Laplace eigenvalue  $\nu(\nu+1)$. Therefore the Robin spectrum admits a systematic double multiplicity, and we remove it beforehand by insisting that $m\geq 0$, resulting in a ``desymmetrized spectrum''.
 Let $\lambda_n(0)$ denote the ordered desymmetrized Neumann eigenvalues (repeated with appropriate multiplicity), and for $\sigma>0$ we denote by $\lambda_n(\sigma)$ the ordered  desymmetrized Robin eigenvalues, and define the Robin-Neumann (RN) gaps by
\[
d_n(\sigma):=\lambda_n(\sigma)-\lambda_n(0).
\]
These were recently investigated in \cite{RWY} in the case of planar domains, and will be the main object of study here.

\subsection{Clusters}
We show that the desymmetrized Robin spectrum breaks up into small clusters $\dsE_\ell(\sigma)$ of size $\lfloor \ell/2 \rfloor +1$, 
concentrated around the Neumann eigenvalues $\ell(\ell+1)$:
 For each eigenvalue $\nu(\nu+1)$, there is some $m\geq 0$, and a corresponding eigenfunction $ e^{im\phi}\legP_\nu^m(\cos \theta)$, so that  the ``degree'' $\nu$ satisfies a secular equation $S_m(\nu)=\sigma$ , where
 \[
 S_m(\nu) =  2\tan \left(\frac{\pi(m+\nu)}2\right)
  \frac{\Gamma\left(\frac{\nu+ m}2+1\right)\Gamma\left(\frac{\nu-m}2+1\right) }  { \Gamma\left(\frac{ \nu+m+1}{2}\right)
  \Gamma\left(\frac{ \nu -m +1}{2}\right)  } .
 \]
 For any integer $\ell\geq m$ of the same parity ($\ell=m\bmod 2$), there is a unique solution $\nu_{\ell,m}(\sigma)$ in the open interval $(\ell,\ell+1)$. Denote by $\Lambda_{\ell,m}(\sigma) = \nu_{\ell,m}(\sigma)(\nu_{\ell,m}(\sigma)+1)$ the resulting Laplace eigenvalue. Then the desymmetrized spectrum consists of   $\Lambda_{\ell,m}(\sigma)$, with $0\leq m\leq \ell$, and $m=\ell \bmod 2$, and is partitioned into disjoint clusters of size $\lfloor \ell/2 \rfloor +1$: 
\[
\dsE_\ell(\sigma)=\{\Lambda_{\ell,m}(\sigma) : 0\leq m\leq \ell,  m=\ell \bmod 2\} .
\]

We denote by $d_{\ell,m}(\sigma)$ the Robin-Neumann (RN) gaps in each cluster:
$$d_{\ell,m}(\sigma) =\Lambda_{\ell,m}(\sigma)-\ell(\ell+1) .$$
We have an asymptotic formula:
\begin{prop}
Fix $\sigma>0$. Let $0\leq m\leq \ell$, $m=\ell \bmod 2$.  If  $\ell-m\to \infty$  then
   \begin{equation}\label{dlm formula}
  d_{\ell,m}(\sigma) \sim  \frac{2\sigma}{\pi}\frac{2\ell+1}{\sqrt{\ell^2-m^2}} .
 \end{equation}
 \end{prop}
We display a plot of these RN gaps in Figure~\ref{fig:clustergaps150}.

\subsection{A Szeg\H{o} type limit theorem}
 We   show, using \eqref{dlm formula},  that the RN gaps from each cluster have a limiting distribution, supported on the ray $[4\sigma/\pi,\infty)$:
 \begin{cor}\label{thm:Szego}
Fix $f\in C_c^\infty(0,\infty)$.  As $\ell \to \infty$,
\[
  \frac 1{\#\dsE_\ell(\sigma)} \sum_{\lambda_n(\sigma) \in \dsE_\ell(\sigma)} f\left( d_n\left( \sigma \right) \right)
=  \int_{4\sigma/\pi}^\infty f(y) \frac{16\sigma^2 dy}{\pi^2 y^3\sqrt{1-(\frac {4\sigma}{\pi y})^2}} .
\]
 \end{cor}
Similarly, we can compute the mean value of the RN gaps within each cluster (\S~\ref{sec:equidistribution and mean}):
\begin{equation}\label{mean value in clusters}
\lim_{\ell \to \infty} \frac 1{\#\dsE_\ell(\sigma)} \sum_{\lambda_n(\sigma) \in \dsE_\ell(\sigma)} d_n(\sigma) \sim 2\sigma .
\end{equation}
Note that $2 = 2\length(\partial \Omega)/\area(\Omega)$, and the general theory\footnote{Strictly speaking, the results of \cite{RWY} are only for planar domains.} developed in \cite{RWY} leads to  \eqref{mean value in clusters} if we average over the entire spectrum.
Finer than that,  \eqref{mean value in clusters} asserts that for the hemisphere the same mean result holds in each cluster.

     \begin{figure}[ht]
\begin{center}
\includegraphics[height=50mm]{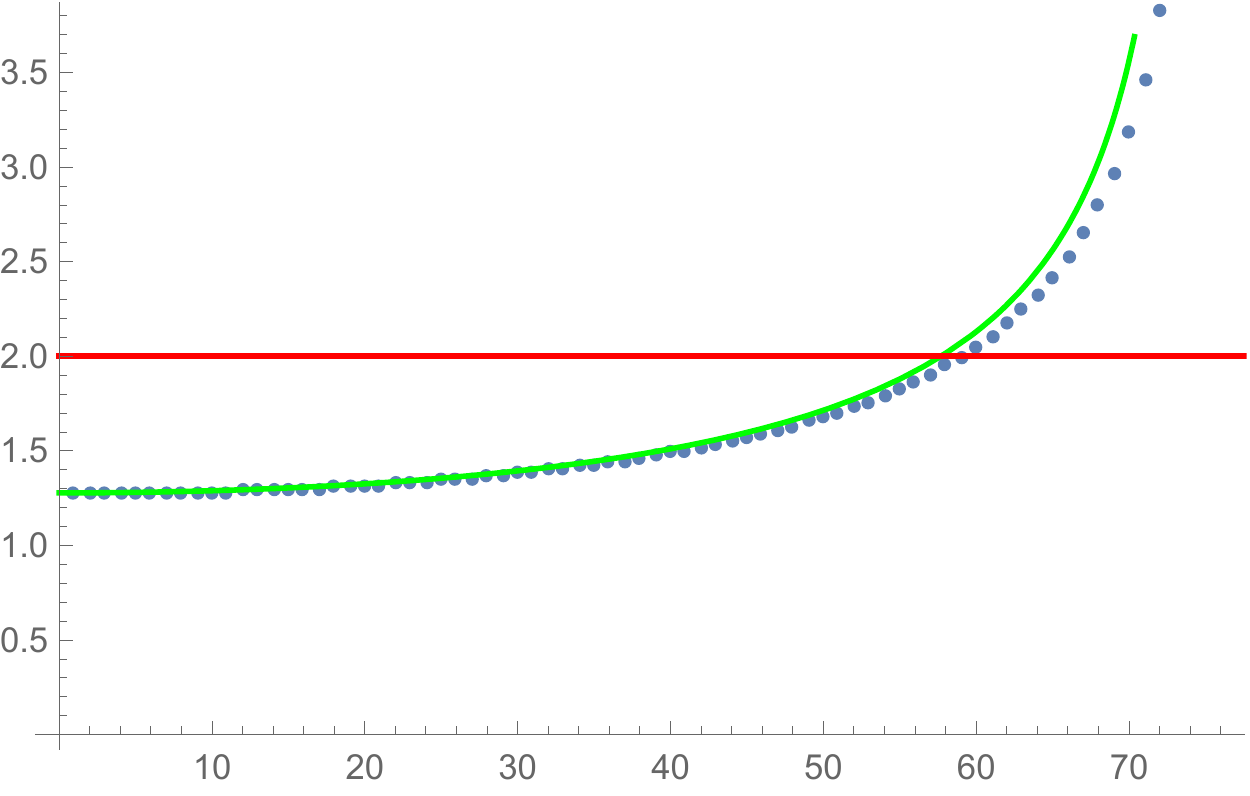}
\caption{The RN differences $d_{\ell,m}(\sigma)$ in the cluster $\dsE_\ell(\sigma)$ for $\ell=150$ and $\sigma=1$. The horizontal line (red) is their mean value $2$. The solid curve (green) is the theoretical formula \eqref{dlm formula}.  }
\label{fig:clustergaps150}
\end{center}
\end{figure}

The cluster structure that we find is similar in nature to that found for the spectrum of  operator $-\Delta+V$ on the unit sphere\footnote{Similar results are available for the spectrum of the Laplace Beltrami operator on Zoll surfaces, which are spheres equipped with a Riemannian metric  such that every geodesic is closed, and all geodesics have the same length. } $S^2$, for a smooth potential $V$ \cite{WeinsteinDuke, Widom}.
 The eigenvalues of  $-\Delta+V$  fall into clusters $C_\ell$ of diameter $O(1)$ around the eigenvalues $\ell(\ell+1)$ of the sphere   (in our case, the   clusters are  bigger, of diameter $\approx \sqrt{\ell}$), 
 and moreover the eigenvalues in each cluster $C_\ell$  become equidistributed  with respect to a suitable measure.

 \subsection{RN gaps}
We next examine the totality of the Robin-Neumann gaps $ d_{n}(\sigma):=\lambda_n(\sigma)-\lambda_n(0)$.

\begin{thm}
\label{thm:NR gaps bound}
 There are constants $0<c<C$ so that for each $\sigma>0$,
\begin{enumerate}[(a)]
\item \label{thm:NR gaps bound part a}  For all $n$,
 \[
 \lambda_n(\sigma)-\lambda_n(0) \leq C \lambda_n(0)^{1/4} \cdot \sigma .
\]
\item  \label{thm:NR gaps bound part b} There are arbitrarily large $n$ so that
  \[
 \lambda_n(\sigma)-\lambda_n(0) \geq c \lambda_n(0)^{1/4} \cdot \sigma .
\]
\end{enumerate}
\end{thm}
In particular the Robin-Neumann gaps for the hemisphere are unbounded. We note that at this point, we do not know of any planar domain where the RN gaps are provably unbounded \cite{RWY}.
The upper bound is better than what is known for general smooth planar domains \cite{RWY}, which is $d_n(\sigma)\le C\lambda_n(0)^{1/3}\cdot \sigma$. 

 As a corollary to Theorem \ref{thm:NR gaps bound} we establish the limit {\em level spacing distribution} for the Robin spectrum, which is the distribution $P(s)$ (assuming it exists) of the nearest-neighbour gaps $\lambda_{n+1}(\sigma)-\lambda_n(\sigma)$, normalized to have mean unity (cf \S~\ref{sec:spacings}).
In the case of the Neumann spectrum on the hemisphere, most of the nearest neighbour gaps $\lambda_{n+1}(0)-\lambda_n(0)$ are zero and $P(s)$ is the delta function at the origin. We show that the Robin spectrum has the same level spacing distribution;
\begin{cor}\label{cor:spacings}
For every $\sigma> 0$, the level spacing distribution for the desymmetrized
 Robin spectrum on the hemisphere  is a delta-function at the origin.
\end{cor}

  However, unlike in the Neumann or Dirichlet case, the delta function is not a result of multiplicities, as there are none here:
\begin{thm}\label{thm:nomult Robin intro}
Fix $\sigma>0$. Then the desymmetrized Robin spectrum is simple: $\lambda_m(\sigma)\neq \lambda_n(\sigma)$ for all $n \neq m$.
\end{thm}
We note that there are few deterministic simplicity results available, unlike generic simplicity which is more common, e.g. the Dirichlet spectrum of generic triangles is simple \cite{HilJudge}. For instance,  simplicity of the desymmetrized Dirichlet spectrum on the disk was proved by Siegel in 1929 (Bourget's hypothesis) \cite{Siegel}, and the same result holds for the Neumann spectrum \cite{Ashu}. However, there are arbitrarily small $\sigma>0$ for which the Robin spectrum on the disk has multiplicities \cite{Yesha}. For the square, we have a result analogous to Theorem~\ref{thm:nomult Robin intro} for $\sigma$ sufficiently small, but for rectangles with irrational squared aspect ratio, it fails for arbitrarily small $\sigma$ \cite{RWrectangles}.

Finally, we note that the theory developed here for the hemisphere is quite singular when compared to what we expect to hold for all other spherical caps. In that case we do not expect a cluster structure and moreover,   we believe that the level spacing distribution will be Poissonian ($P(s)=\exp(-s)$), as is expected for most integrable systems \cite{BerryTabor, ZRwhatisQC}, compare Figure~\ref{fig:cappiover3upto100spacings}.

\section{The Robin problem}
\subsection{Basics}
Denote by $\Omega$ the upper hemisphere  on the unit sphere, given in spherical coordinates as
\[
\Omega=\Big\{(\sin \theta \cos \varphi , \sin \theta \,\sin \varphi , \cos \theta)  : 0\leq \phi<2\pi , \;0\leq \theta\leq \pi/2 \Big\}
\]
so that the north pole is at $\theta=0$, and the equator, which is the boundary $\partial\Omega$, is at $\theta=\pi/2$.

 We consider the Robin boundary problem on the hemisphere $\Omega$:
$$
\Delta F+\nu(\nu+1)F=0, \quad  \frac{\partial F}{\partial n} +\sigma F=0
$$
with $\nu>0$, where $\partial/\partial n$ is the derivative in the direction of the outward pointing normal to the equator, and $\sigma>0$. We will call $\nu$ the ``degree'',
in keeping with the case of Dirichlet or Neumann boundary conditions, when the eigenfunctions are spherical harmonics of degree $\ell$, with eigenvalue $\ell(\ell+1)$.

 For $\sigma>0$, all  eigenvalues   $ \lambda=\nu(\nu+1)$   are positive, hence  $\nu$ is real and $\nu>0$ or $\nu<-1$.
Since the two solutions of $ \lambda=\nu(\nu+1)$ are $\nu$ and $-1-\nu$, we may assume that $\nu>0$.

The Laplacian commutes with rotations, hence the problem admits a separation of variables, according to symmetry under rotations $\{R_\phi\}$ around the north-south pole,
which defines ``sectors'' consisting of functions transforming as $F(R_\phi x) = e^{im\phi} F(x)$ (here $m\in \Z)$.
We  write such a Robin eigenfunction as
$$
F(\phi,\theta) = e^{im\phi } f_{\nu,m}(\cos \theta)
$$
where $f(x)$ is a solution of ($x:=\cos\theta$) 
\begin{equation} \label{form 1 of Legendre eq}
(1-x^2)f'' -2xf'+\left( \nu(\nu+1)-\frac{m^2}{1-x^2}\right) f =0 .
  \end{equation}
The Robin boundary condition $\sigma F +\frac{\partial F}{\partial n}=0$ is then translated to
\begin{equation}\label{robin bdry cond}
 \sigma f(0)-f'(0)=0  .
\end{equation}
Indeed, the equator is $\theta=\pi/2$, or $x=0$; and the normal derivative  (outward normal) is
$$
\frac{\partial}{\partial n}\Big|_{\theta=\pi/2}=-\frac{d}{dx}\Big|_{x=0} .
$$

 \subsection{Desymmetrization}
 Since the equation \eqref{form 1 of Legendre eq} is independent of the sign of $m$, we see that the Robin spectrum has a systematic double multiplicity. We will remove it (desymmetrization) by insisting that $m\geq 0$.
Note that this is equivalent to taking only eigenfunctions which are symmetric with respect to the reflection $(x,y,z)\mapsto (x,-y,z)$.
 We order the desymmetrized Neumann eigenvalues (including multiplicities) by
\[
\lambda_0=0<\lambda_1 =2<\lambda_2=\lambda_3=6<\dots
\]

\subsection{The eigenfunctions}
 The solutions of the differential equation \eqref{form 1 of Legendre eq} which are nonsingular  in   $0\leq x\leq 1$  form a one-dimensional space,
 all multiples of the associated Legendre functions (Ferrers functions) of the first kind $\legP_\nu^m$ \cite[14.3.4]{DLMF}
 \begin{equation}\label{rep of P as Ohyp}
 \begin{split}
\legP_\nu^m(x) &= (-1)^m \frac{\Gamma(\nu+m+1)}{2^m\Gamma(\nu-m+1)} (1-x^2)^{m/2}
\Ohyp  \left(\nu+m+1,m-\nu;m+1,\frac{1-x}2  \right)
\\
&=  (-1)^m \frac{\Gamma(\nu+m+1)}{2^m\Gamma(\nu-m+1)}
\left( \frac{1-x}{1+x}\right)^{m/2} \Ohyp\left(\nu+1,-\nu;m+1;\frac{1-x}2\right) .
\end{split}
\end{equation}
 Here  $\Ohyp\left(a,b;c;z\right)$ is Olver's hypergeometric series
 \[
 \Ohyp \left(a,b;c;z\right)=\sum_{s=0}^\infty \frac{(a)_s(b)_s}{\Gamma(c+s)s!} z^s, \quad |z|<1
 \]
 with $(a)_s=\Gamma(a+s)/\Gamma(a)$,
 so that \eqref{rep of P as Ohyp}   converges absolutely if   $x \in (-1,1]$, in particular in the range $x=\cos \theta\in [0,1]$ which is relevant for the hemisphere.

\section{The secular equation}
For integer $m\geq 0$, we set
\begin{equation}\label{sec eq m}
S_m(\nu) =  2\tan \left(\frac{\pi(m+\nu)}2\right)
  \frac{\Gamma\left(\frac{\nu+ m}2+1\right)\Gamma\left(\frac{\nu-m}2+1\right) }  { \Gamma\left(\frac{ \nu+m+1}{2}\right)
  \Gamma\left(\frac{ \nu -m +1}{2}\right)  } .
\end{equation}
Plots of $S_4(\nu)$ and $S_5(\nu)$ are displayed in figure~\ref{fig:seceqm5}.
\begin{figure}[ht]
\begin{center}
\includegraphics[height=50mm]{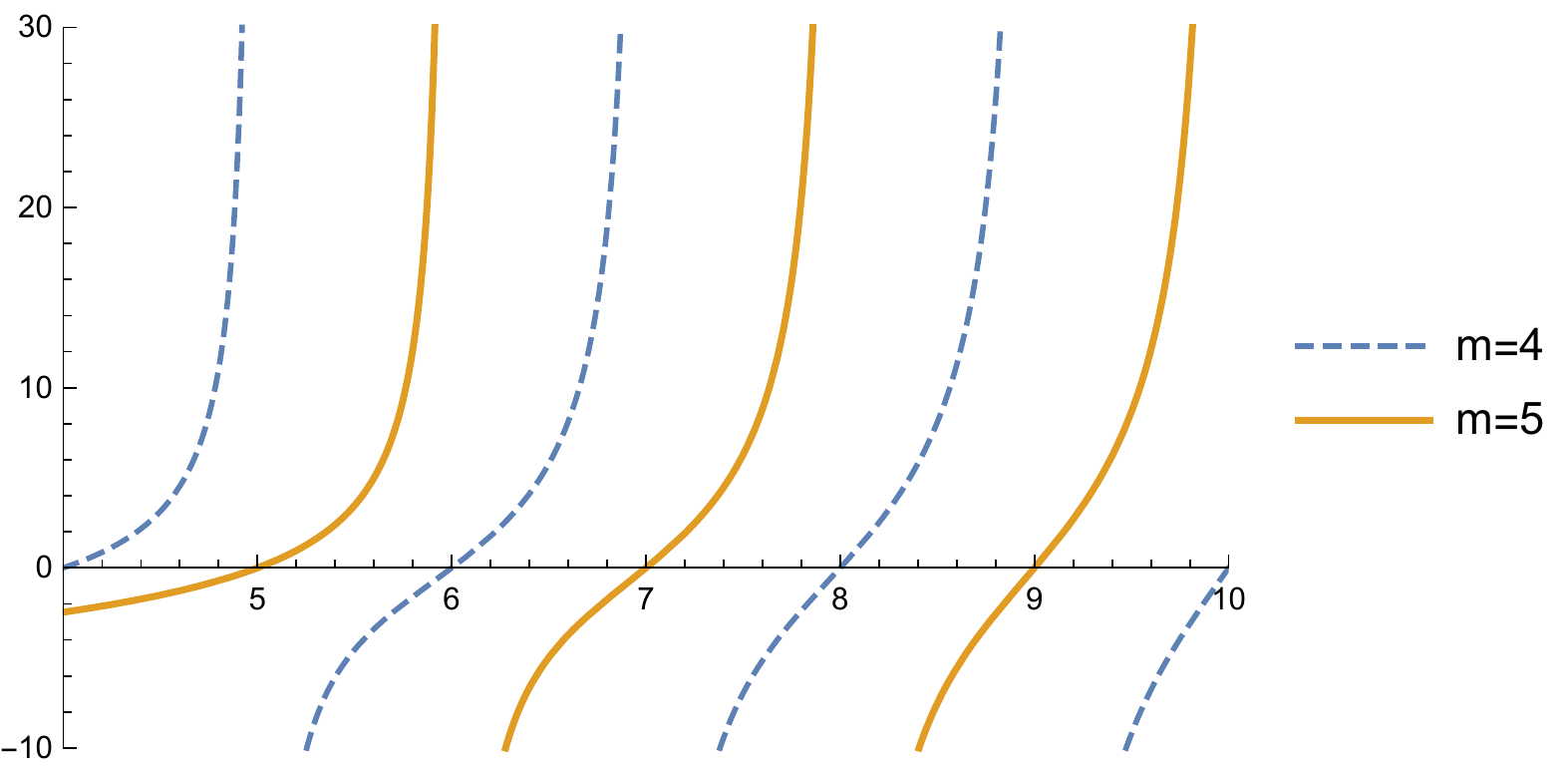}
\caption{ $S_4(\nu)$ (dashed) and $S_5(\nu)$ (solid). }
\label{fig:seceqm5}
\end{center}
\end{figure}



\begin{thm}\label{thm:def of nu}
Let $\sigma>0$.
\begin{enumerate}[(a)]
\item  For each $m\geq 0$, the degree $\nu>0$ for which the boundary value problem \eqref{form 1 of Legendre eq} and \eqref{robin bdry cond}   admits nonzero regular solutions satisfies
the secular equation
\begin{equation*}
S_m(\nu) = \sigma .
\end{equation*}

\item The secular equation  has no solutions in $0<\nu<m$.

\end{enumerate}
\end{thm}

\begin{proof}
We saw that for all $\nu$, there is a one-dimensional space of solutions of the ODE \eqref{form 1 of Legendre eq} which are regular for $x\in [-1,1]$, spanned by the associated Legendre function $\legP_\nu^m(x)$.
The  boundary condition \eqref{robin bdry cond} gives the secular equation
\begin{equation*}
\frac{f_{\nu,m}'(0)}{f_{\nu,m}(0)} = \frac{ \left(\frac {d\legP_\nu^m}{dx}\right)(0)}{ \legP_\nu^m(0) } =
\sigma .
\end{equation*}
The values at $x=0$ of $\legP_\nu^m$ and its derivative are \cite[\S 14.5 (i)]{DLMF} 
\[
\legP_\nu^m(0)   =
\frac{ 2^m \sqrt{\pi}}{\Gamma\left(\frac{\nu-m}{2}+1\right)\Gamma\left(\frac{1-\nu-m}{2}\right)}
= \frac{2^{m}}{\sqrt{\pi}} \cos\left(\frac{\pi(\nu+m)}2\right) \frac{\Gamma\left(\frac{\nu+m+1}{2}\right)}{\Gamma\left(\frac{\nu-m}2+1\right)}
\]
and
\[
\left(\frac {d\legP_\nu^m}{dx}\right)(0)= -\frac{2^{m+1}\sqrt{\pi}}{\Gamma\left(\frac{\nu-m+1}{2}\right)\Gamma\left(-\frac{\nu+m}{2}\right)}
= \frac{2^{m+1}}{\sqrt{\pi}}\sin\left(\frac{\pi(m+\nu)}2\right)\frac{\Gamma\left(\frac{\nu+m}2+1\right)} { \Gamma\left(\frac{\nu-m+1}{2}\right)}
\]
and therefore
\[
\frac{ \left(\frac {d\legP_\nu^m}{dx}\right)(0)}{ \legP_\nu^m(0) }
= 2\tan \left(\frac{\pi(m+\nu)}2\right)
  \frac{\Gamma\left(\frac{\nu+ m}2+1\right)\Gamma\left(\frac{\nu-m}2+1\right) }{ \Gamma\left(\frac{ \nu+m+1}{2}\right)\Gamma\left(\frac{ \nu -m +1}{2}\right)  } .
\]
Hence we obtain the secular equation in the form $S_m(\nu)=\sigma$ with $S_{m}$ as in \eqref{sec eq m}.


We transform $S_m(\nu)$ by using Euler's reflection formula $\Gamma(s)\Gamma(1-s) = \pi/\sin(\pi s)$ to convert
\[
\begin{split}
 \frac{ \Gamma(\frac{\nu-m}2+1) }{   \Gamma(\frac{ \nu -m +1}{2})  } & =
\left(   \frac{\pi}{\sin ( \frac{\pi(m-\nu)}{2}) \Gamma(\frac{m-\nu}{2})}\right)/
\left(  \frac{\pi}{\sin\frac{\pi (\nu-m+1)}{2} \Gamma(1-\frac{\nu-m+1}{2})     }\right)
\\
&=  \frac{ \Gamma( \frac{m-\nu+1}2) }  {  \Gamma(\frac{m-\nu}{2}) }
\cdot \frac{\cos (\pi \frac{m-\nu}{2}) }{ \sin (\pi \frac{m-\nu}{2}) }
= \frac{ \Gamma( \frac{m-\nu+1}2) }  {  \Gamma(\frac{m-\nu}{2}) } \cot\left(\pi\frac{m-\nu}{2}\right)  .
\end{split}
\]
Moreover, for integer $m$,
\[
\tan \left(\frac{\pi(m+\nu)}2\right)\cdot  \cot\left(\frac{ \pi(m-\nu)}{2}\right)  = -1 .
\]
Thus we obtain
\begin{equation} \label{sec eq m neq 0}
    S_m(\nu) = -2 \frac{\Gamma(\frac{\nu+ m}2+1) \Gamma( \frac{m-\nu+1}2) }  {\Gamma(\frac{ m+\nu+1}{2})  \Gamma(\frac{m-\nu}{2}) }  .
 \end{equation}
The expression  \eqref{sec eq m neq 0}  allows us check that if $m\geq 1$, there is no solution for the secular equation if  $0<\nu<m$  (recall $\sigma>0$),
because the arguments of all the Gamma functions on the r.h.s. of \eqref{sec eq m neq 0} are positive
if $0<\nu<m$, hence so are the Gamma functions.
Therefore $S_m(\nu)$ is  {\em negative} for $\nu<m$. Thus for $0<\nu<m$ there is no solution of the secular equation if $\sigma>0$.
 \end{proof}


\begin{prop}\label{prop:bound on delta}
Fix $\sigma>0$. Then
\begin{enumerate}[(a)]
\item
$S_m$ vanishes at the points $m+2k$, with $k\geq 0$ integer, tends to infinity as $\nu\nearrow m+2k+1$, and
$S_m(\nu)$ is negative for $m+2k-1<\nu<m+2k$, positive in $m+2k<\nu<m+2k+1$ and increasing for $m+2k-1<\nu < m+2k+1$.

\item \label{def of nu(l,m)}
Let $\ell = m+2k$ with integer $k=0,1,2,\dots$. Then there is a unique solution $ \nu_{\ell,m}(\sigma)\in (\ell,\ell+1)$ of the secular equation.
\item
  Write $\nu_{\ell,m}(\sigma) = \ell+\delta_{\ell,m}(\sigma)$, with $\delta=\delta_{\ell,m}(\sigma)\in (0,1)$. Then
\begin{equation}\label{delta small}
  \delta <\frac{ \sqrt{\frac 2\pi}\sigma}{\sqrt{\nu}} .
  \end{equation}
\end{enumerate}
\end{prop}

\begin{proof}
We use $S_{m}$ in the form
 \begin{equation*}
S_m(\nu) =  2\tan \left(\frac{\pi(m+\nu)}2\right)  G(\nu+m)G(\nu-m)
\end{equation*}
where
\[
 G(s):= \frac{\Gamma(\frac{s}{2}+1)}{ \Gamma(\frac{s+1}{2})} .
 \]
 Note that $G(s)$ is positive for $s>0$.
 We have for $s>0$,
  \[
  G'(s) = \frac 12 G(s)\left( \psi\left( \frac s2 +1\right)-\psi\left(  \frac s2 +\frac 12\right) \right)
  \]
  with $\psi$ the digamma function \cite[5.9.16]{DLMF}
  \[
  \psi(s) :=\frac{\Gamma'(s)}{\Gamma(s) } =-\gamma+ \int_0^1\frac {    1-t^{s-1}}{1-t} dt , \quad \Re(s)>0
  \]
  so that
  \[
\frac{G'(s)}{G(s)} =\frac 12   \int_0^1 \frac{(1-t^{s/2})-(1-t^{(s-1)/2})}{1-t} dt =  \frac 12  \int_0^1 \frac{t^{(s-1)/2}}{1+\sqrt{t}} dt
  \]
  is clearly positive for $s>0$. Since $G(s)>0$ we deduce that $G'(s)>0$ for $s>0$, so that $G(s)$ is increasing, and
\begin{equation}
\label{eq:G'/G<1/2}
0< \frac{G'(s)}{G(s)} < \frac{1}{2}.
\end{equation}

The function $S_m(\nu)$ is positive for $m+2k<\nu<m+2k+1$ because both $G(\nu\pm m)$ are positive for $\nu>m$, and writing
$\nu = m+2k+\delta$ gives $\tan\frac{\pi}{2}(m+\nu)=\tan \frac{\pi}{2}\delta$ which is positive for $\delta\in (0,1)$, and negative for
$\delta\in (-1,0)$.

 The logarithmic derivative of $S_m$ is
 \begin{equation}
 \label{eq:Sm log der sum}
 \frac{S_m'}{S_m}(\nu) = \frac{\pi}{\sin \pi \delta} + \frac{G'}{G}(\nu-m) + \frac{G'}{G}(\nu+m).
 \end{equation}
 Since $G'/G>0$, we find that if $\delta\in (0,1)$ then $S_m'/S_m(\nu)>0$ and since $S_m(\nu)>0$ for all $v>m$ we obtain that $S_m'(\nu)>0$ for $\nu\in (m+2k,m+2k+1)$, so that $S_m$ is increasing there. Otherwise, if $\nu \in (m+2k-1,m+2k)$, then $\delta\in (-1,0)$, and we already know
 that here $S_{m}(\nu)<0$. Then, since in this range $\frac{\pi}{\sin \pi \delta} < -\pi$, the inequality \eqref{eq:G'/G<1/2} shows, with the
 use of the triangle inequality, that the r.h.s. of \eqref{eq:Sm log der sum} is
 \begin{equation*}
 \frac{\pi}{\sin \pi \delta} + \frac{G'}{G}(\nu-m) + \frac{G'}{G}(\nu+m) < -\pi+ \frac{1}{2}+\frac{1}{2} < 0,
 \end{equation*}
and so is the l.h.s. of \eqref{eq:Sm log der sum}, and then $S_m'(\nu) >0$.

Since $G(s)$ is positive and increasing, for $\nu>m$ we get
  \[
  G(\nu-m)\geq G(0) =\frac 1{\sqrt{\pi}} .
  \]
  By Stirling's formula
$ G(s)\sim \sqrt{\frac s2} +O(1/\sqrt{s})$ as $s\to \infty$, in fact \cite[5.6.4]{DLMF}
\begin{equation}\label{Gautschi Inequality}
\sqrt{\frac s2}<G(s)<\sqrt{\frac s2+1}, \quad s>0.
\end{equation}
   Also note
  \[
\tan \left(\frac{\pi(m+\nu)}2\right)  = \tan \pi\left(m+k+\frac{\delta}{2}\right)  = \tan \frac{\pi \delta}{2} \geq  \frac{\pi \delta}{2} .
\]
  We obtain
  \[
  \sigma  = 2\tan (\frac{\pi(m+\nu)}2)  G(\nu+m)G(\nu-m) > 2\frac{\pi \delta}{2}  \sqrt{\frac{\nu+m}{2}}G(0) \geq  \delta \sqrt{\frac \pi 2}\sqrt{\nu}
  \]
  so that $  \delta < \sqrt{\frac 2\pi}\sigma/\sqrt{\nu}$.
 \end{proof}

\begin{cor}
 Fix $\sigma>0$. For $\ell\geq m \geq 0$, $\ell=m\bmod 2$, let $\nu = \nu_{\ell,m}(\sigma)$ be the unique solution of the secular equation $S_m(\nu) = \sigma$ with $\nu\in (\ell,\ell+1)$.
 Write $\nu = \ell+\delta$, with $\delta\in (0,1)$. Then
\begin{enumerate}[a.]
\item As $\sigma \to 0$, $\delta\to 0$,

\item As $\sigma\to \infty$, we have $\delta\to 1$.
\end{enumerate}
  \end{cor}

Consequently, as $\sigma\to 0$, $\nu\to \ell $, while as $\sigma\to \infty$, $\nu\to \ell+1$. Thus, as $\sigma$ varies between $0$ and $+\infty$, $\Lambda_{\ell,m}(\sigma):=\nu_{\ell,m}(\sigma)\cdot (\nu_{\ell,m}(\sigma)+1)$ interpolates between a Neumann eigenvalue $\ell(\ell+1)$ with $\ell$ of the same parity as $m$, and a Dirichlet eigenvalue $(\ell+1)(\ell+2)$ with same $m$ and opposite parity between $\ell$ and $m$.
\begin{proof}
 That $\delta\to 0$ as $\sigma\to 0$ follows from \eqref{delta small}.
 Using monotonicity of $G(s)$ we obtain
  \[
  \sigma=S_m(\nu) \leq 2 \tan \frac{\pi \delta}{2} G(2m+2k+1)G(2k+1) \ll_{m,k} \tan \frac{\pi \delta}{2}
  \]
  so that as $\sigma\to \infty$, we have $\delta\to 1$.
  \end{proof}

 \section{Multiplicity one}\label{sec:multiplicity one}

We have seen (Theorem~\ref{thm:def of nu}) that the desymmetrized Robin spectrum of the hemisphere is given by the energies
\begin{equation}\label{def of Lambda}
 \Lambda_{\ell,m}(\sigma) = \nu_{\ell,m}(\sigma)\cdot (\nu_{\ell,m}(\sigma)+1)
 \end{equation}
 with $\ell\ge 0$, and $0\le m\le \ell$ satisfying $m\equiv \ell \mod 2$, satisfying  the secular equation $S_m(\nu)=\sigma$, with $S_m$ given by
\eqref{sec eq m}:
\begin{equation*}
S_m(\nu) =  2\tan \left(\frac{\pi(m+\nu)}2\right)
  \frac{\Gamma(\frac{\nu+ m}2+1)\Gamma(\frac{\nu-m}2+1) }  { \Gamma(\frac{ \nu+m+1}{2})  \Gamma(\frac{ \nu -m +1}{2})  } .
\end{equation*}

To show that there are no degeneracies in the desymmetrized spectrum (Theorem~\ref{thm:nomult Robin intro}), it therefore suffices to prove:
\begin{prop}
\label{prop:nomult Robin}
Fix $\sigma>0$. For all $\ell\ge 2$ and $0\le m\le \ell-2$ with $m\equiv \ell \bmod 2$,
\begin{equation*}
\nu_{\ell,m+2}(\sigma)>\nu_{\ell,m}(\sigma).
\end{equation*}
\end{prop}

\begin{figure}[ht]
\begin{center}
\includegraphics[height=60mm]{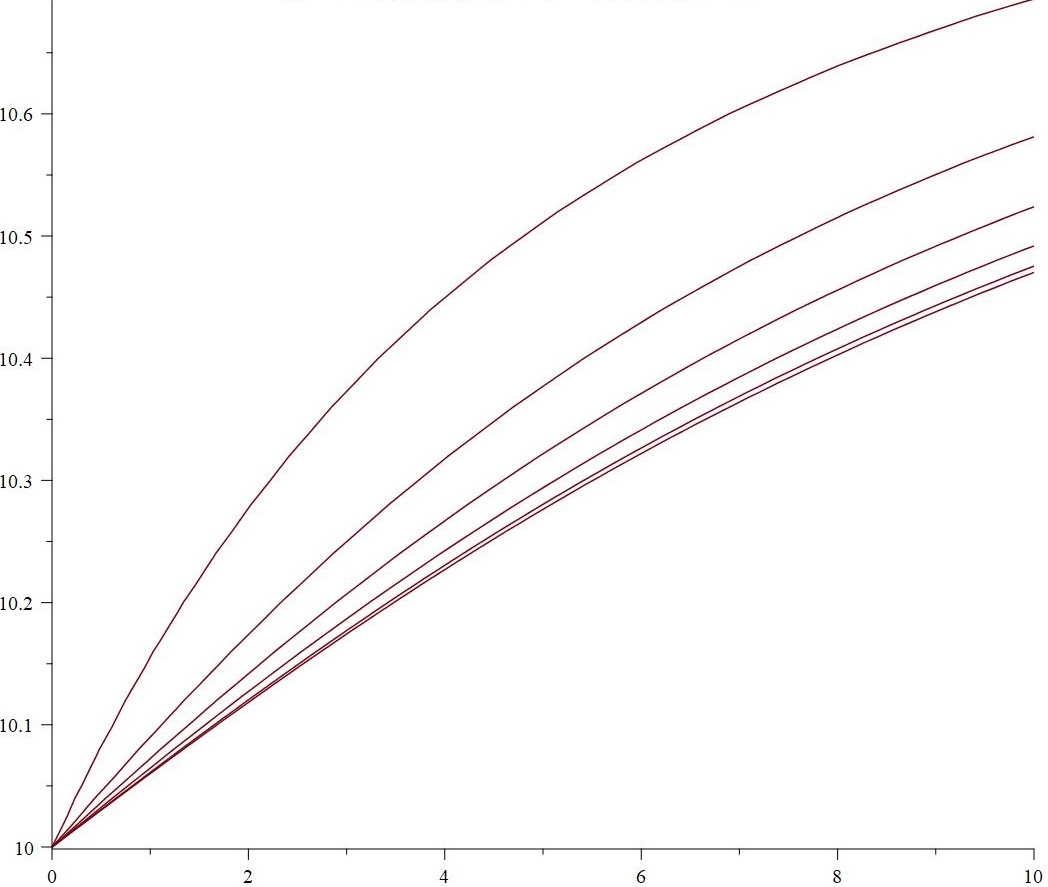}
\caption{Plots of $\nu_{10,m}$, $m=0,2,4,6,8,10$ on $[0,10]$. As asserted by Proposition \ref{prop:nomult Robin},
higher curves correspond to larger value of $m$.}
\label{fig:nu10m 10}
\end{center}
\end{figure}

The picture emerging for $\nu_{10,m}(\sigma)$ on $[0,10]$, with all possible $0\le m\le 10$, $m\equiv \ell\mod{2}$,
is displayed within figures \ref{fig:nu10m 10}. This clearly support the statement of Proposition \ref{prop:nomult Robin}.

\begin{proof}


Recall that $\nu_{\ell,m}(\sigma) \in (\ell,\ell+1)$, and that $\ell=m+2k$, $k\geq 0$.
By Proposition~\ref{prop:bound on delta}, both $S_m(\nu)$ and $S_{m+2}(\nu)$ are increasing and positive in $(\ell,\ell+1)$.
Using the recurrence $\Gamma(s+1)=s\Gamma(s)$ we find
\begin{equation*}
\frac{ S_{m+2}(\nu)}{S_m(\nu)}  = \frac{ \frac{\nu+m}{2}+1 \cdot \frac{\nu-m-1}{2}}{ \frac{\nu+m+1}{2} \cdot \frac{\nu-m}{2}} = 1-\frac{2(m+1)}{(\nu-m)(\nu+m+1)}<1.
\end{equation*}
Hence for $\nu\in (\ell,\ell+1)$, where both
$S_m(\nu)$ and $S_{m+2}(\nu)$ are positive, we must have $S_{m+2}(\nu)<S_{m}(\nu)$.
Therefore
$$S_{m+2}(\nu_{\ell,m}(\sigma)) < S_{m}(\nu_{\ell,m}(\sigma))=\sigma
=S_{m+2}(\nu_{\ell,m+2}(\sigma)).$$
Since $S_{m+2}$ is increasing in $(\ell,\ell+1)$, we deduce that $\nu_{\ell,m}(\sigma) < \nu_{\ell,m+2}(\sigma)$ as claimed.
 \end{proof}

 \section{Clusters and a Szeg\H{o} type limit theorem}\label{sec:clusters}
\subsection{Cluster structure}
Denote the cluster (a multiset) of desymmetrized multiple Neumann eigenvalues sharing a common value of $\ell(\ell+1)$ by
\[
\dsE_\ell(0)=\Big\{\ell(\ell+1): 0\leq m\leq \ell, m=\ell \mod 2  \Big\}.
\]
This cluster has size $\#\dsE_\ell(0) =\lfloor \ell/2 \rfloor +1$. We label the eigenvalues there by
\[
\dsE_\ell(0) = \left\{ \lambda_L,\lambda_{L+1},\dots,\lambda_{L+\lfloor \ell/2\rfloor } \right\}
\]
where $L=L_\ell$ is given by
\[
L=\#\left( \dsE_0(0) \cup \dsE_1(0)\cup\dots \cup \dsE_{\ell-1}(0)\right) = \sum_{\ell'=0}^{\ell-1} \left\lfloor \frac {\ell'}2 \right\rfloor+1 = \frac{\ell^2}{4}+O(\ell).
\]

The distance of the Neumann eigenvalue cluster $\dsE_\ell(0)$ to the closest other Neumann eigenvalue cluster, which for $\ell\geq 1$ is $\dsE_{\ell-1}(0)$ (in other words, the distance between distinct nearby Neumann eigenvalues), is
\begin{equation}\label{dist between clusters at 0}
\min_{ \ell' : \ell'\neq \ell}\dist \Big( \dsE_\ell(0), \dsE_{\ell'}(0)\Big)  = \ell(\ell+1)-(\ell -1)\ell = 2\ell .
\end{equation}

We saw that the Robin eigenvalues are $\nu(\nu+1)$ where $\nu=\nu_{\ell,m}(\sigma)\in (\ell,\ell+1)$, $\ell=m\bmod 2$,
is a solution of the secular equation $S_m(\nu) = \sigma$.
Denote by
\begin{equation}
\label{eq:clusters def}
\dsE_\ell(\sigma) = \{\Lambda_{\ell,m}(\sigma):\ell\geq m\geq 0, \ell  = m\bmod 2 \}
\end{equation}
which is the evolution of the  Neumann eigenvalue cluster  $\dsE_\ell(0)$.  Since $\ell<\nu_{\ell,m}(\sigma) < \ell+1$, the spectral
cluster  $\dsE_\ell(\sigma)$ is contained in the  open interval $\Big( \ell(\ell+1), (\ell+1)(\ell+2) \Big)$,
and in particular the evolved eigenvalue clusters $\dsE_\ell(\sigma)$  do not mix with each other.

 \subsection{Asymptotics of the Robin-Neumann gaps}\label{sec:mean value}


Recall that we write $\nu_{\ell,m}(\sigma) = \ell +\delta_{\ell,m}(\sigma)$.
\begin{lem}
As $\ell \to \infty$, with $0\leq m< \ell$, $\ell=m\bmod 2$, 
\begin{equation}\label{asymp of delta}
  \delta_{\ell,m}(\sigma) = \frac{ 2\sigma}{\pi \sqrt{\ell^2-m^2}}\left(1+O\left( \frac 1{\ell-m} \right) \right).
\end{equation}
For $m=\ell$, we have
\begin{equation}\label{asymp of delta m=ell} 
\delta_{\ell,\ell}(\sigma)  \sim \frac{\sigma}{\sqrt{\pi \ell}}.
\end{equation}
 
\end{lem}
\begin{proof}
For $0<\ell-m=O(1)$, \eqref{asymp of delta}  is just the upper bound \eqref{delta small}, so  assume $\ell-m\to \infty$.
The cluster $\dsE_\ell(\sigma)$ consists of $\lfloor \ell/2\rfloor +1$ eigenvalues $\Lambda_{\ell,m}(\sigma) = \nu_{\ell,m}(\nu_{\ell,m}+1)$ with  $ m+2k=\ell$, $m,k\geq 0$, and where $\nu_{\ell,m}(\sigma)$ is the unique solution of the secular equation $S_m(\nu)=\sigma$ in the interval $(\ell,\ell+1)$.
We write
$$\nu=\nu_{\ell,m}(\sigma)= \ell+\delta = m+2k+\delta, \qquad \delta=\delta_{\ell,m}(\sigma).$$

Recall that the $S_{m}$ of the secular equation $S_{m}(\nu)=\sigma$ is given by
\begin{equation}\label{sec eq m again}
S_m(\nu) =  2\tan \left(\frac{\pi(m+\nu)}2\right)  G(\nu+m)G(\nu-m)
\end{equation}
where $ G(s)= \Gamma(\frac{s}{2}+1)/ \Gamma(\frac{s+1}{2})$  satisfies (cf. \eqref{sec eq m})
\[
G(s) = \sqrt{\frac s2} \left(1+O\left(\frac 1{s}\right) \right),\quad s\to \infty.
\]
Since we assume that $ \ell-m=2k\to \infty$,  both arguments of $G$ in \eqref{sec eq m again} tend to infinity, because
$\nu+m = 2k+2m+\delta=\ell+m+\delta$ and $\nu-m = 2k+\delta =\ell-m+\delta$.
Moreover,
\[
\tan\frac \pi 2 (\nu+m) = \tan \frac \pi 2 \delta
\]
and we  know (Proposition~\ref{prop:bound on delta}) that
 \begin{equation}\label{delta small b}
  \delta\ll \sigma/\sqrt{\ell} \to 0
  \end{equation}
so that
\[
\tan\frac \pi 2 (\nu+m) = \tan \frac \pi 2 \delta =\frac \pi 2 \delta +O\left(\frac 1{\ell^{3/2}}\right) .
\]
Therefore we can write
\begin{equation}\label{asymp of Sm}
\begin{split}
S_m(\nu) &= 2 \frac \pi 2 \delta\left( 1+O\left( \frac 1\ell\right)\right) \cdot
 \sqrt{\frac{\nu-m}2}\left( 1 +O\left( \frac 1{ \ell-m } \right) \right)
 \\
 & \qquad \cdot  \sqrt{\frac{\nu+m}2} \left( 1 +O\left( \frac 1{ \ell+m } \right) \right)
\\
& = \pi   \delta \cdot \sqrt{k+\frac \delta 2}\sqrt{k+m+\frac \delta 2} \left( 1 +O\left( \frac 1{\ell-m} \right)\right) .
\end{split}
\end{equation}
Furthermore, since $2k=\ell-m$,
\[
 \sqrt{k+\frac \delta 2} = \sqrt{k} \left( 1+O\left( \frac \delta{\ell-m} \right) \right)  =
 \sqrt{\frac{\ell-m}2} \left( 1+O\left( \frac 1{\sqrt{\ell}(\ell-m)} \right) \right)
\]
and likewise since $k+m=(\ell +m)/2$
\[
\sqrt{k+m+\frac \delta 2}  = \sqrt{\frac{\ell+m}2} \left( 1+O\left( \frac 1{\ell^{3/2}}  \right) \right).
\]

Inserting \eqref{asymp of Sm} into the secular equation $S_m(\nu) = \sigma$ gives, when $ \ell-m\to \infty$, that
\begin{equation*}
  \delta_{\ell,m}(\sigma) 
   = \frac{2\sigma}{\pi\sqrt{\ell^2-m^2}} \left(1+O\left( \frac 1{\ell-m} \right) \right) .
\end{equation*}

\ 
When $m=\ell$, we use $\delta=\delta_{\ell,\ell}(\sigma)\ll 1/\sqrt{\ell} \to 0$ and $G(0)=1/\sqrt{\pi}$ to obtain
\[
\sigma=S_\ell(\sigma) =2\tan \left(\frac{\pi}{2}\delta\right) G(2\ell+\delta)G(\delta)\sim \pi \delta G(2\ell)G(0) \sim \pi \delta \sqrt{\ell} \frac 1{\sqrt{\pi}}
\]
as $\ell \to \infty$, which gives \eqref{asymp of delta m=ell}. 
 \end{proof}

 We derive an asymptotic for the RN gaps $d_{\ell,m}(\sigma)  = \Lambda_{\ell,m}(\sigma)-\ell(\ell(+1)$ in each cluster:

\begin{cor}
As $\ell \rightarrow \infty$, for all $0\le m < \ell$ with $m=\ell \bmod 2$, 
the Robin-Neumann gaps satisfy
 \begin{equation}\label{dlm formula 2}
  d_{\ell,m}(\sigma) = \frac{2\sigma}{\pi}\cdot \frac{2\ell+1}{\sqrt{\ell^2-m^2}} +O\left(  \frac{\sqrt{\ell}}{(\ell-m)^{3/2}}\right) .
 \end{equation}
 For $m=\ell$ we have
  \begin{equation}\label{dlm formula m=ell}
  d_{\ell,\ell}(\sigma) \sim \frac{2\sigma}{\sqrt{\pi}}\sqrt{\ell}. 
   \end{equation}
\end{cor}
\begin{proof}
We have
\begin{multline*}
d_{\ell,m}(\sigma) = \Lambda_{\ell,m}(\sigma) - \Lambda_{\ell,m}(0) = (\nu-\ell)(\nu+\ell+1)=\delta(2\ell+1+\delta)
\\
= (2\ell+1)\delta_{\ell,m} + \delta_{\ell,m}^2  =  (2\ell+1)\delta_{\ell,m} +O\left(\frac 1{\ell}\right)
\end{multline*}
where we have used \eqref{delta small b}.
Moreover, for $  m<\ell$ we have the asymptotic formula \eqref{asymp of delta} for $\delta_{\ell,m}$, and hence
\[
\begin{split}
d_{\ell,m}(\sigma)  &=\frac{ 2(2\ell+1)\sigma}{\pi \sqrt{\ell^2-m^2}}\left(1+O\left( \frac 1{\ell-m} \right) \right) +O\left(\frac 1{\ell}\right)
\\
&=\frac{ 2(2\ell+1)\sigma}{\pi \sqrt{\ell^2-m^2}} +O\left(  \frac{\sqrt{\ell}}{(\ell-m)^{3/2}} + \frac 1\ell \right)
\\
&=\frac{ 2(2\ell+1)\sigma}{\pi \sqrt{\ell^2-m^2}} +O\left(  \frac{\sqrt{\ell}}{(\ell-m)^{3/2}}\right) .
\end{split}
\]
For the case $m=\ell$, \eqref{dlm formula m=ell}   similarly follows from \eqref{asymp of delta m=ell}. 
\end{proof}

\subsection{Equidistribution of gaps in the cluster}\label{sec:equidistribution and mean}
 We can now deduce the equidistribution of gaps in each cluster (Corollary~\ref{thm:Szego}) and compute the average gap in a cluster as asserted in \eqref{mean value in clusters}. Since the arguments are similar, we do the latter:
\begin{cor}
As $\ell\to \infty$,
\[
\frac 1{\#\dsE_\ell(\sigma)} \sum_{\lambda_n(\sigma) \in \dsE_\ell(\sigma)} d_n(\sigma) \sim 2\sigma .
\]
\end{cor}

\begin{proof}
Using $d_{\ell,m} = (2\ell+1)\delta_{\ell,m}+\delta_{\ell,m}^2\ll \sqrt{\ell}$ by \eqref{delta small}, we see that we may restrict the average to $m\leq  \ell- 1$ with an error of $O(\ell^{-1/2})$:
\[
\frac 1{\#\dsE_\ell(\sigma)} \sum_{\lambda_n(\sigma) \in \dsE_\ell(\sigma)} d_n(\sigma)
=
\frac 1{\ell/2+O(1)} \sum_{\substack{ 0\leq m\leq \ell-1\\  m=\ell \bmod 2}} d_{\ell,m} + O(\ell^{-1/2 }) .
\]
Then we  use \eqref{dlm formula 2} to obtain
\[
\frac 1{\ell/2+O(1)} \sum_{\substack{0\leq m\leq \ell-1\\m=\ell \bmod 2}} d_{\ell,m}
  = \frac 1{\ell /2}\sum_{\substack{0\leq m\leq \ell-1\\m=\ell \bmod 2}}
\frac{ 2(2\ell+1)\sigma}{\pi \sqrt{\ell^2-m^2}}
+O\left( \frac 1{\ell^{1/2}} \right) .
\]
Moreover, using standard bounds for the rate of convergence of Riemann sums gives
\[
\begin{split}
\frac 1{\ell /2}\sum_{\substack{0\leq m\leq \ell-1\\m=\ell \bmod 2}}
\frac{ 2(2\ell+1)\sigma}{\pi \sqrt{\ell^2-m^2}}
&=\left(\frac{4 \sigma}{\pi} +O(\frac 1\ell)\right)  \left(
\int_0^{1 } \frac {dx}{\sqrt{ 1-x^2}}  +O\left(\frac 1{\ell^{ 1 /2}} \right)
\right)
\\
&= 2\sigma + O\left(\frac 1{\ell^{1/2}} \right)    .
\end{split}
\]
Altogether, we obtain
\[
\frac 1{\#\dsE_\ell(\sigma)} \sum_{\lambda_n(\sigma) \in \dsE_\ell(\sigma)} d_n(\sigma) =
2\sigma+ O\left(\frac 1{\ell^{1/2}}   \right)    \sim 2\sigma
\]
as claimed.
\end{proof}

\section{Bounds for the RN gaps: Proof of Theorem~\ref{thm:NR gaps bound}}

\begin{proof}
 Using \eqref{delta small} shows that for $\ell\gg 1$ and $0\le m \le \ell$ with $m\equiv \ell \mod{2}$,
 \[
 \Lambda_{\ell,m}(\sigma)-\ell(\ell+1) = (\nu_{\l,m}(\sigma)-\ell)(\nu_{\ell,m}(\sigma)+\ell+1) \ll \sigma\sqrt{\ell}
 \]
 so that
 \[
 \max\Big\{ |\lambda-\ell(\ell+1)|:  \lambda\in \dsE_\ell(\sigma) \Big\} \ll  \sigma \sqrt{\ell}  .
\]
Therefore, for all $n $, we  have
 \begin{equation}\label{cor:bound on distance}
  \lambda_n(\sigma)-\lambda_n(0) \ll \sigma \lambda_n(0)^{1/4} .
  \end{equation}
 This proves Theorem~\ref{thm:NR gaps bound}\ref{thm:NR gaps bound part a}.


To show  that we can actually attain the upper bound in \eqref{cor:bound on distance}, note that Proposition \ref{prop:nomult Robin}
demonstrates that to get the largest possible Robin-Neumann gaps, it is worth, given $\ell\ge 0$, to take $m=\ell$.   
We then use \eqref{dlm formula m=ell} to obtain
\[
  d_{\ell,\ell}(\sigma) \sim \frac{2\sigma}{\sqrt{\pi}}\sqrt{\ell} \sim \frac{2}{\sqrt{\pi}} \Lambda_{\ell,\ell}(0)^{1/4} \sigma ,
  \]
 which proves Theorem~\ref{thm:NR gaps bound}\ref{thm:NR gaps bound part b}. 
 \end{proof}



We note that $\Lambda_{\ell,\ell}(0)\in \dsE_\ell(0)$, and therefore for each $\ell\gg 1$,  we have found
$n=\ell^2/4+O(\ell)$ for which
 \[
 \lambda_n(\sigma)-\lambda_n(0) \gg \lambda_n(0)^{1/4} \cdot \sigma,
\]
 and in particular that the Robin-Neumann gaps are unbounded.


  \section{Level spacings}\label{sec:spacings}
In this section, we show that the level spacing distribution of the desymmetrized  Robin spectrum on the hemisphere is a delta function at the origin, as is the case with Neumann or Dirichlet boundary conditions.
We note that for other spherical caps (cf \cite{Haines} for background), we expect that the level spacing distribution is Poissonian.
A numerical plot for the desymmetrized Dirichlet spectrum on the cap with opening angle $\theta_0=\pi/3$ (the hemisphere has $\theta_0=\pi/2$) is displayed in Figure~\ref{fig:cappiover3upto100spacings}.

\begin{figure}[ht]
\begin{center}
\includegraphics[height=40mm]{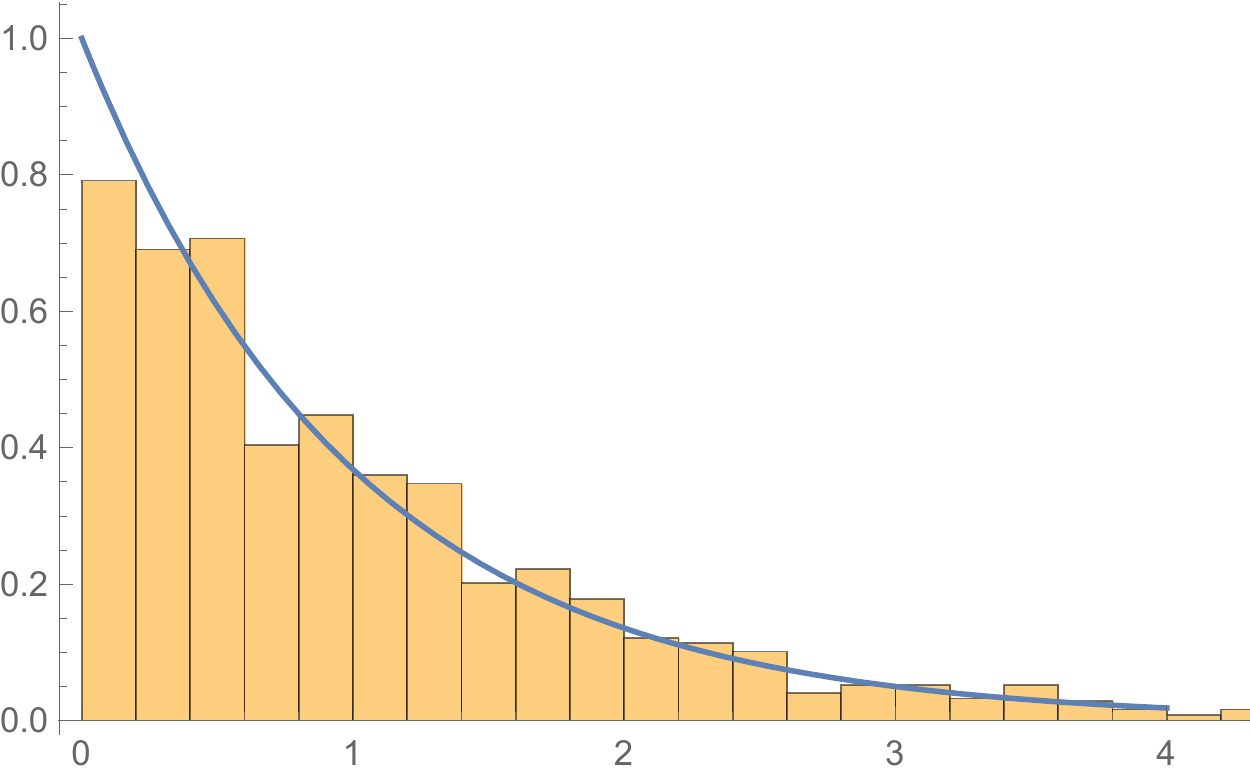}
\caption{The level spacing distribution $P(s)$ for all $1258$ desymmetrized Dirichlet eigenvalues $\nu(\nu+1)$ with $\nu<100$ for the spherical cap with opening angle $\theta_0=\pi/3$. The solid curve is the Poisson result $\exp(-s)$.}
\label{fig:cappiover3upto100spacings}
\end{center}
\end{figure}

 \begin{proof}[Proof of Corollary \ref{cor:spacings}]
 The statement of Corollary~\ref{cor:spacings} 
is equivalent to the fact that for every $y>0$,
\begin{equation}\label{eq:large gaps}
\lim\limits_{N\rightarrow\infty}\frac{1}{N}\#\{ n\le N:\: \lambda_{n+1}^{\sigma}-\lambda_{n}^{\sigma}>y\} = 0.
\end{equation}

Recall that we divided  the ordered desymmetrized Robin eigenvalues $\{\lambda_{n}^{\sigma}\}_{n\ge 0}$ into disjoint clusters $\dsE_\ell(\sigma)$ (see \eqref{eq:clusters def}), each at distance $O(\sqrt{\ell})$ from the Neumann eigenvalues $\ell(\ell +1)$, so $\diam \dsE_\ell(\sigma)\ll \sqrt{\ell}$ (Theorem~\ref{thm:NR gaps bound}\ref{thm:NR gaps bound part a}), and hence of distance $2\ell +O(\sqrt{\ell})$ from the closest other cluster,
and of size $\#\dsE_\ell(\sigma) = \lfloor \ell/2\rfloor +1 = \ell/2+O(1)$.

For $N\gg 1$, denote by $L$ the index of the cluster to which $\lambda_N(\sigma)$ belongs, so that
$$\bigcup\limits_{\ell \leq L-1} \dsE_\ell(\sigma) \subset \{\lambda_n(\sigma):n\leq N\}\subseteq \bigcup\limits_{\ell \leq L} \dsE_\ell(\sigma)$$
and therefore
$$N=\sum\limits_{\ell\leq L-1} \#\dsE_\ell(\sigma) +O(L) = \frac{L^2}{4} +O(L)$$
so that
 $L=O(\sqrt{N})$.
Then
\begin{equation}
\label{eq:numb large gaps sum clust}
\#\{ n\le N:\: \lambda_{n+1}^{\sigma}-\lambda_{n}^{\sigma}>y\} \leq \sum\limits_{\ell=0}^{L}
\sum\limits_{\substack{\lambda_{n+1}^{\sigma}-\lambda_{n}^{\sigma}>y \\ \lambda_{n}^{\sigma}\in \dsE_{\ell}(\sigma)}}1 .
\end{equation}

Denote by $n_+$ the maximal index of an eigenvalue in $\dsE_\ell(\sigma)$, and by $n_-$ the minimal index. Then the gaps corresponding to the cluster $\dsE_\ell(\sigma)$ are firstly those with $\lambda_{n+1}(\sigma)-\lambda_n(\sigma)$ with $n_-\leq n\leq n_+-1$ and secondly, the last gap $\lambda_{n_++1}(\sigma)-\lambda_{n_+}(\sigma)$.
The number of those gaps of the second kind is at most $L+1 = O(\sqrt{N})$.

For the gaps $>y$ of the first kind, we have in each cluster
\begin{equation*}
\begin{split}
\sum\limits_{\substack{\lambda_{n+1}^{\sigma}-\lambda_{n}^{\sigma}>y \\ \lambda_{n}^{\sigma}\in \dsE_{\ell}(\sigma) \\ n<n_+}}1
&<
\sum\limits_{\substack{\lambda_{n+1}^{\sigma}-\lambda_{n}^{\sigma}>y \\  n_-\leq n<n_+}}  \frac{\lambda_{n+1}(\sigma)-\lambda_n(\sigma)}{y}
\\
&\leq
\sum\limits_{ n_-\leq n<n_+}  \frac{\lambda_{n+1}(\sigma)-\lambda_n(\sigma)}{y}
= \frac{\lambda_{n_+}-\lambda_{n_-}}{y} .
\end{split}
\end{equation*}
Now $\lambda_{n_+}-\lambda_{n_-} = \diam \dsE_\ell(\sigma) \ll \sqrt{\ell}$, and so we find that
\begin{equation}
\label{eq:numb gaps telescope}
\sum\limits_{\substack{\lambda_{n+1}^{\sigma}-\lambda_{n}^{\sigma}>y \\ \lambda_{n}^{\sigma}\in \dsE_{\ell}(\sigma) \\ n<n_+}}1 \ll \frac{\sqrt{\ell}}y .
\end{equation}
Summing the inequality \eqref{eq:numb gaps telescope}
over $\ell \leq L = O(\sqrt{N})$ gives
\[
\sum\limits_{\ell=0}^{L}\sum\limits_{\substack{\lambda_{n+1}^{\sigma}-\lambda_{n}^{\sigma}>y \\ \lambda_{n}^{\sigma}\in \dsE_{\ell}(\sigma) \\ n<n_+}}1
\ll \sum_{\ell \leq L} \frac{\sqrt{\ell}}y  \ll \frac{L^{3/2}}y \ll \frac{N^{3/4}}y .
\]
Altogether, substituting this into \eqref{eq:numb large gaps sum clust},
and upon taking into account the gaps of the second kind, we find that for $N\gg_y 1$,
\[
\#\{ n\le N:\: \lambda_{n+1}^{\sigma}-\lambda_{n}^{\sigma}>y\} \ll \frac{N^{3/4}}y +\sqrt{N},
\]
which proves \eqref{eq:large gaps}.
\end{proof}

\end{document}